\documentclass[11pt, draft]{amsart}

\makeatletter

\theoremstyle{plain}
\newtheorem{thm}{Theorem}
\newtheorem{cor}[thm]{Corollary} 
\newtheorem{lem}[thm]{Lemma} 
\theoremstyle{definition}

\theoremstyle{remark}

\theoremstyle{remark}

\theoremstyle{remark}

\theoremstyle{definition}

\theoremstyle{definition}

\theoremstyle{plain}

\theoremstyle{definition}

\theoremstyle{remark}

\theoremstyle{remark}

\theoremstyle{definition}

\newcommand{\p}{\varphi}

\def\freeprod{\font\bigsymbolsfont=cmsy10 scaled \magstep3
 \setbox0=\hbox{\bigsymbolsfont\char'003 }\mathord{\lower1pt\box0}}\relax\ignorespaces

\usepackage{amssymb} \usepackage{amscd}

\makeatother

\begin{document}

\title{Quasitraces are traces: A short proof of the finite-nuclear-dimension case.}

\author{Nathanial P.\ Brown and Wilhelm Winter}

\address{Department of Mathematics, Penn State University, State
College, PA 16802, USA}

\email{nbrown@math.psu.edu} 

\address{School of Mathematical Sciences, University of Nottingham, Nottingham NG7 2RD, UK} 

\email{wilhelm.winter@nottingham.ac.uk} 

\thanks{The first named author is partially supported by DMS-0856197. The second named author is partially supported by EPSRC First Grant EP/G014019/1.}

\begin{abstract}  Uffe Haagerup proved that quasitraces on unital exact C$^*$-algebras are traces.   We give a short proof  under the stronger hypothesis of locally finite nuclear dimension; our result generalizes to the case of lower semicontinuous extended quasitraces on nonunital C$^{*}$-algebras.

Uffe Haagerup a d\'emontr\'e qu'une quasi-trace sur une C*-alg\`ebre
exacte \`a \'el\'ement unit\'e est une trace. Nous donnons une
courte d\'emonstration sous l'hypoth\`ese plus forte de dimension
nucl\'eaire localement finie; ce r\'esultat se g\'en\'eralise
jusqu'au cas d'une quasi-trace \'etendue semicontinue inf\'erieurement
sur une C*-alg\`ebre sans \'el\'ement unit\'e. 
\end{abstract}

\maketitle

\noindent
In 1991 Uffe Haagerup distributed a hand written manuscript \cite{haagerup:quasitraces} containing a proof of the following theorem: Every quasitrace on a unital, exact C$^*$-algebra is a trace.\footnote{In this note, ``trace" means a positive, linear functional satisfying the trace condition $\tau(ab) = \tau(ba)$ and ``quasitrace" is synonymous with ``2-quasitrace" as defined in \cite{BH} (in particular, a quasitrace is finite everywhere). We use the term ``extended trace'' (or ``extended quasitrace'', respectively) when the value $+\infty$ is allowed on the positve cone.}  
We will give a very short proof\footnote{In fact, the first published proof. In \cite{Uffe-Steen}, it is  shown that every state on K$_0(A)$ arises from a tracial state, when $A$ is unital and exact. In \cite{kirchberg} it is shown that every simple, stably projectionless, exact C$^*$-algebra has a densely defined trace.  And \cite{Blanchard-Kirchberg} shows that every quasitrace on a non-unital exact C$^*$-algebra is a trace, via reduction to  Haagerup's unpublished paper \cite{haagerup:quasitraces}. None of this published work covers our result.} for  quasitraces on (not necessarily unital) C$^{*}$-algebras with locally finite nuclear dimension (in the sense of \cite{WZ}), a case that covers all simple C$^{*}$-algebras that have been classified by their Elliott invariants so far; in fact, at this point we do not know any example of a nuclear C$^{*}$-algebra which does not have locally finite nuclear dimension.  We generalize this result to lower semicontinuous extended quasitraces on C$^{*}$-algebras with locally finite nuclear dimension.

Our first three lemmas are well known, but we isolate them for convenience.  The proof of the first one reduces to the fact that a quasitrace on a full matrix algebra must take the same value on all minimal projections (since they're all Murray-von Neumann equivalent). In the following, $A$ and $B$ are C$^{*}$-algebras, not necessarily unital.

\begin{lem} 
\label{lem:fdcase}  Quasitraces on finite-dimensional C$^*$-algebras are traces. 
\end{lem}

\begin{lem} 
\label{lem:orderzero} If $\p \colon A \to B$ is a completely positive order zero map (cf.\ \cite{wz}) and $\tau$ is a  trace (resp.\ quasitrace) on $B$, then $\tau\circ \p$ is a  trace (resp.\ quasitrace) on $A$. 
\end{lem} 
\begin{proof} This is Corollary 3.4 in \cite{wz}. 
\end{proof} 

\begin{lem} 
\label{lem:RadonNikodym}
Let $\tau$ be a quasitrace on $A$ and assume there is a trace $\tau'$ on $A$ such that $\tau \leq \tau'$ (i.e., $\tau(x) \leq \tau'(x)$ for all $x \in A_+$).  Then $\tau$ is a  trace. 
\end{lem} 

\begin{proof} This follows from Corollaries II.1.9 and II.2.4 in \cite{BH}. \end{proof} 





Here is the main result.

\begin{thm} 
\label{main}
If $A$ has nuclear dimension $n$, then every quasitrace $\tau$ on $A$ is a  trace.   
\end{thm} 

\begin{proof}  By Lemma \ref{lem:RadonNikodym}, it suffices to show the existence of a  trace $\tau'$ on $A$ such that $\tau \leq \tau'$. 

According to \cite[Proposition 3.2]{WZ} we can find nets of completely positive maps $\psi_\lambda \colon A \to \oplus_{i=0}^n F^{(i)}_\lambda$ and $\p_\lambda \colon \oplus_{i=0}^n F^{(i)}_\lambda \to A$ such that each $F^{(i)}_\lambda$ is finite dimensional; $\|a - \p_\lambda \circ \psi_\lambda (a) \| \to 0$ for all $a \in A$; $\p^{(i)}_\lambda := \p_\lambda|_{F^{(i)}_\lambda}$ is contractive with order zero for all $i,\lambda$;  $\| \psi_\lambda \| \leq 1$ for all $\lambda$; and the induced map $$\bar{\psi} \colon A \to \frac{\prod_{\lambda} F_\lambda}{\bigoplus_\lambda F_\lambda}$$ has order zero. 

Since each $\p^{(i)}_\lambda$ has order zero, $\tau \circ \p^{(i)}_\lambda$ is a  trace on $F^{(i)}_{\lambda}$ (combine Lemmas \ref{lem:fdcase} and \ref{lem:orderzero}). Since $\tau$ is bounded \cite[Corollary II.2.3]{BH}, $\sup_\lambda \| \tau \circ \p^{(i)}_\lambda \| < \infty$ and hence, for each free ultrafilter $\omega$ on the index set $\Lambda$, we get a  trace $(\tau \circ \p^{(i)}_\lambda)_\omega$ on $$\frac{\prod_{\lambda} F^{(i)}_\lambda}{\bigoplus_\lambda F^{(i)}_\lambda}$$ by taking the limit of the $\tau \circ \p^{(i)}_\lambda$'s along $\omega$. Letting $$\bar{\psi_i} \colon A \to \frac{\prod_{\lambda} F^{(i)}_\lambda}{\bigoplus_\lambda F^{(i)}_\lambda}$$ be the induced order-zero map, the composition $(\tau \circ \p^{(i)}_\lambda)_\omega \circ \bar{\psi_i}$ is a  trace on $A$.   

The final fact we need is ``2-subadditivity" of  quasitraces, i.e., the fact that $\tau(x + y) \leq 2(\tau(x) + \tau(y))$ for all $x,y \in A_+$ (see \cite[Corollary II.2.5]{BH}). An induction argument shows the (far from sharp) inequality $\tau(\sum_0^n x_i) \leq 2^n \sum_0^n \tau(x_i)$ for positive $x_{i}$, so the following completes the proof: 
\begin{align*}
\tau(x) & = \lim_\omega \tau( \p_\lambda \circ \psi_\lambda (x) )\\
&\leq 2^n \sum_{i=0}^n  \lim_\omega \tau( \p^{(i)}_\lambda \circ \psi^{(i)}_\lambda (x) )\\
&=  2^n \sum_{i=0}^n (\tau \circ \p^{(i)}_\lambda)_\omega \circ \bar{\psi_i} (x).
\end{align*}
\end{proof} 

Since  quasitraces are norm-continuous \cite[Corollary II.2.5]{BH}, the following generalization is easily deduced.  

\begin{cor} If $A$ has locally  finite nuclear dimension (i.e., every finite set is nearly contained in a subalgebra that has finite nuclear dimension), then every  quasitrace on $A$ is a  trace.  
\end{cor} 


We are indebted to George Elliott for drawing our attention to the question whether the preceding results generalize to extended quasitraces. The answer is affirmative in the lower semicontinuous case. 

\begin{thm}
\label{rems}
If $A$ has locally finite nuclear dimension, then every  lower semicontinuous extended quasitrace $\tau$ on $A$ is an extended trace.
\end{thm}

\begin{proof}
Let us first assume that $A$ has finite nuclear dimension. 

It will suffice to show that $\tau(a+b)=\tau(a)+\tau(b)$ for all $a,b \in A_{+}$ for which $\tau(a)$ and $\tau(b)$ are both finite. 

Observe that $a+b$ is Cuntz subequivalent to $a \oplus b$ in the sense of \cite{BH}. By \cite[Proposition~2.4]{Ror:UHFII},  for any  $\epsilon>0$  there are $\delta>0$ and  $x \in M_{2,1}(A)$ such that $(a+b - \epsilon)_{+} = x^{*} ((a - \delta)_{+} \oplus (b-\delta)_{+}) x$.  By \cite{BH} this implies that 
\begin{eqnarray*}
d_{\tau}((a+b - \epsilon)_{+}) & \le &  d_{\tau}((a - \delta)_{+} \oplus (b-\delta)_{+}) \\
& = & d_{\tau}((a - \delta)_{+}) + d_{\tau}( (b-\delta)_{+}) \\
& \le & \frac{1}{\delta} (\tau(a) +\tau(b))\\
&  < & \infty,
\end{eqnarray*} 
where $d_{\tau}$ denotes the lower semicontinuous dimension function associated with $\tau$. 

Again by \cite{BH} for any positive contraction $c \in B_{\epsilon}:= \overline{((a+b-2\epsilon)_{+}) A ((a+b-2\epsilon)_{+})}\subset A$ we have $\tau(c) \le d_{\tau}((a+b-\epsilon)_{+})< \infty$, whence $\tau$ restricts to a  (bounded) quasitrace $\tau_{\epsilon}$ on $B_{\epsilon}$. Since $B_{\epsilon}$ has  finite nuclear dimension by \cite[Proposition~2.5]{WZ}, we see from Theorem~\ref{main} that every $\tau_{\epsilon}$ is in fact a  trace on $B_{\epsilon}$. 

Let $h_{\epsilon} \in \mathcal{C}_{0}(0,1]$ be the positive  function which is $0$ on $(0,2\epsilon]$, $1$ on $[3\epsilon,1]$, and linear on $[2\epsilon,3\epsilon]$, and set $d_{\epsilon}:= h_{\epsilon}(a+b) \in A$; note that the $d_{\epsilon}$ form an approximate unit for $\overline{(a+b)A(a+b)} \subset A$ and that $B_{\epsilon} = \overline{d_{\epsilon}  A d_{\epsilon}}$. By lower semicontinuity of $\tau$, and since each $\tau_{\epsilon}$ is a trace on $B_{\epsilon}$, we have
\begin{eqnarray*}
\tau(a+b) & = & \lim_{\epsilon \to 0} \tau(d_{\epsilon} (a+b)d_{\epsilon}) \\
& = & \lim_{\epsilon \to 0} \tau_{\epsilon}(d_{\epsilon} (a+b)d_{\epsilon}) \\
& = & \lim_{\epsilon \to 0} \tau_{\epsilon}(d_{\epsilon} a d_{\epsilon}) + \lim_{\epsilon \to 0} \tau_{\epsilon}(d_{\epsilon} bd_{\epsilon}) \\
& = & \lim_{\epsilon \to 0} \tau(d_{\epsilon} a d_{\epsilon}) + \lim_{\epsilon \to 0} \tau(d_{\epsilon} bd_{\epsilon}) \\
& = & \tau(a) + \tau(b),
\end{eqnarray*}
as desired.

Next, suppose $A$ has only locally finite nuclear dimension, and let $a,b \in A$ be positive contractions. Suppose first  that $\tau(a)$, $\tau(b)$, $d_{\tau}(a)$, $d_{\tau}(b)$ are all finite, say bounded by some $K \ge 1$.

Let $\eta >0$ and choose $0<\epsilon<\eta/8K$ such that 
\[
\tau((a+b-4\epsilon)_{+}) \ge \tau(a+b) - \eta, \, \tau((a-2\epsilon)_{+}) \ge \tau(a) - \eta \mbox{ and } \tau((b-2\epsilon)_{+}) \ge \tau(b) - \eta;
\]
this is possible by lower semicontinuity of $\tau$. Since $A$ has locally finite nuclear dimension, there are a C$^{*}$-subalgebra $B \subset A$ with finite nuclear dimension and positive contractions $a',b' \in B$ such that
\[
\|a' - a\|, \, \|b'-b\| <\epsilon.
\]
From \cite[Lemma~2.2]{KR2} we see that there are contractions $x,y,z,r,s,v,w \in A$ such that 
\[
(a'-\epsilon)_{+} = x^{*} a x,\,   (b'-\epsilon)_{+} = y^{*} b y, \, (a'+b' - 2\epsilon)_{+} = z^{*} (a+b)z,
\]
\[
(a-2 \epsilon)_{+} = r^{*} (a'-\epsilon)_{+} r, \, (b-2\epsilon)_{+} = s^{*} (b'-\epsilon) s, \, (a+b - 4 \epsilon)_{+} = v^{*} (a'+b'-2\epsilon)_{+} v,
\]
\[
 ((a'-\epsilon)_{+} + (b'-\epsilon)_{+} - 4\epsilon)_{+} = w^{*}(a+b)w. 
\]

We estimate
\begin{eqnarray*}
\tau(a+b) & \ge & \tau(((a'-\epsilon)_{+} +(b'-\epsilon)_{+} -4\epsilon)_{+}) \\
&\ge& \tau((a'-\epsilon)_{+} + (b'-\epsilon)_{+}) - 4 \epsilon \cdot  d_{\tau}((a'-\epsilon)_{+} + (b'-\epsilon)_{+}) \\
& \ge &  \tau((a'-\epsilon)_{+} + (b'-\epsilon)_{+}) - 4 \epsilon \cdot  (d_{\tau}((a'-\epsilon)_{+}) + d_{\tau}((b'-\epsilon)_{+})) \\ 
& \ge & \tau((a'-\epsilon)_{+} + (b'-\epsilon)_{+}) - 4 \epsilon \cdot  (d_{\tau}(a) + d_{\tau}(b)) \\ 
& \ge & \tau((a'-\epsilon)_{+} + (b'-\epsilon)_{+}) - \eta\\
& \ge & \tau((a' + b'-2\epsilon)_{+}) - \eta\\
& \ge & \tau((a + b-4\epsilon)_{+}) - \eta\\ 
& \ge & \tau(a + b) - 2\eta,\\ 
\end{eqnarray*}
which implies 
\[
|\tau(a+b) - \tau((a'-\epsilon)_{+} + (b'-\epsilon)_{+}) | <\eta.
\]
Since $\tau|_{B}$ is an extended trace by the first part of the proof, we have 
\[
\tau((a'-\epsilon)_{+}+(b'-\epsilon)_{+}) = \tau((a'-\epsilon)_{+}) + \tau((b'-\epsilon)_{+}).
\] 
Moreover, we have
\[
\tau(a) \ge \tau((a'-\epsilon)_{+}) \ge \tau((a-2\epsilon)_{+}) \ge \tau(a) - \eta
\]
and 
\[
\tau(b) \ge \tau((b'-\epsilon)_{+}) \ge \tau((b-2\epsilon)_{+}) \ge \tau(b) - \eta.
\]
This yields
\[
|\tau((a'-\epsilon)_{+}) + \tau((b'-\epsilon)_{+}) - (\tau(a) +\tau(b))| <  \eta,
\]
hence
\[
|\tau(a+b) - \tau(a) + \tau(b)| < 2 \eta.
\]
Since $\eta>0$ was arbitrary, we obtain
\[
\tau(a+b) = \tau(a) + \tau(b),
\]
at least if $\tau(a)$, $\tau(b)$, $d_{\tau}(a)$ and $d_{\tau}(b)$ are all finite.

If we only assume $\tau(a)$ and $\tau(b)$ to be finite, then for any $\eta>0$ $d_{\tau}((a-\eta)_{+})$ and $d_{\tau}((b-\eta)_{+})$ will be finite, and by the preceding argument we have 
\[
\tau((a-\eta)_{+} +(b-\eta)_{+}) = \tau((a-\eta)_{+}) + \tau((b-\eta)_{+});
\]
by lower semicontinuity this yields 
\[
\tau(a+b) =\lim_{\eta \to 0} \tau((a-\eta)_{+} +(b-\eta)_{+}) = \lim_{\eta \to 0} \tau((a-\eta)_{+}) + \lim_{\eta \to 0} \tau((b-\eta)_{+}) = \tau(a) + \tau(b).
\]
If $\tau(a)$ and $\tau(b)$ are not both finite, there is nothing to show, so we are done.
\end{proof}

\noindent{\textbf{Acknowledgement:} The first author thanks the second, and the other members of Nottingham's math department, for the hospitality during the visit that resulted in this paper. He also thanks the University of Hawaii for their hospitality during the sabbatical year when this work was completed.}

\end{document}